\documentclass[preprint,12pt]{elsarticle}

\usepackage{graphicx}
\usepackage{caption,color}   
\usepackage{amssymb}
\usepackage{amsthm}
\usepackage{amsmath}
\usepackage{epic}
\usepackage{CJK}
\usepackage{setspace}
\usepackage{float}
\usepackage{multirow}
\newtheorem{thm}{Theorem}[section]

\newtheorem{lem}[thm]{Lemma}

\begin{document}
\begin{frontmatter}
\title{\textbf{The reduced formula of the characteristic polynomial of hypergraphs and the spectrum of hyperpaths\\ }}

\author{Changjiang Bu}\ead{buchangjiang@hrbeu.edu.cn}
\author{Lixiang Chen}


\address{College of Mathematical Sciences, Harbin Engineering University, Harbin 150001, PR China}

\begin{abstract}
In this paper, we give a reduced formula of the characteristic polynomial of  $k$-uniform hypergraphs with a pendant edge. And the explicit characteristic polynomial and all distinct eigenvalues of $k$-uniform  hyperpath are given.
\end{abstract}

\end{frontmatter}

\section{Introduction}

For a positive integer $n$, let $\left[ n \right] = \left\{ {1, \ldots ,n} \right\}$. A $k$-order $n$-dimension complex tensor $\mathcal{T} = \left( {{t_{{i_1} \cdots {i_k}}}} \right) $ is a multidimensional array with $n^k$ entries on complex number field $\mathbb{C}$, where ${i_j} \in \left[ n \right]$, $j = 1, \ldots ,k$. Denote the set of $n$-dimension complex vector and the set of $k$-order $n$-dimension complex tensor by $\mathbb{C}^n$ and $\mathbb{C}^{[k,n]}$, respectively. For $x = {\left( {{x_1}, \ldots ,{x_n}} \right)^{\rm{T}}} \in {\mathbb{C}^n}$, $\mathcal{T}{x^{k - 1}}$ is a vector in $\mathbb{C}^{n}$ whose $i$-th component is defined as
\[{\left( {\mathcal{T}{x^{k - 1}}} \right)_i} = \sum\limits_{{i_2}, \ldots ,{i_k}=1}^n {{t_{i{i_2} \cdots {i_k}}}{x_{{i_2}}} \cdots {x_{{i_k}}}} .\]
If there exist $\lambda \in \mathbb{C}$ and a nonzero vector $x =(x_1,x_2,\ldots,x_n)^{\mathrm{T}}\in {\mathbb{C}^n}$ such that $\mathcal{T}{x^{k - 1}} = \lambda {x^{\left[ {k - 1} \right]}}$, then $\lambda$ is called an \emph{eigenvalue} of $\mathcal{T}$ and $x$ is called an \emph{eigenvector} of $\mathcal{T}$ corresponding to $\lambda$, where ${x^{\left[ {k - 1} \right]}} = {\left( {x_1^{k - 1}, \ldots ,x_n^{k - 1}} \right)^{\rm{T}}}$ (see \cite{Lim,Qi}). The \emph{characteristic polynomial }${\phi _\mathcal{T}}\left( \lambda  \right)$ of tensor $\mathcal{T}$ is defined as the resultant ${\rm{Res}} {{{(\lambda {x^{[k - 1]}} - \mathcal{T}{x^{k - 1}})}}}$. And ${\phi _\mathcal{T}}\left( \lambda  \right)$ is a monic polynomial in $\lambda$ of degree $n(k-1)^{n-1}$ (see \cite{cooper xing}).

A hypergraph $H=(V(H),E(H))$ is called \emph{$k$-uniform} if each edge of $H$ contains exactly $k$ distinct vertices. When $k=2$, $H$ is a graph.  The tensor  $\mathcal{A}_H=(a_{i_1i_2\ldots i_k}) \in \mathbb{C}^{[k,n]}$ is the \emph{adjacency tensor} of a $k$-uniform hypergraph $H$ with vertex set $V(H)=\{1,2,\ldots,n \}$, where
\[{a_{{i_1}{i_2} \ldots {i_k}}} = \left\{ \begin{array}{l}
 \frac{1}{{\left( {k - 1} \right)!}},{\kern 37pt}\mathrm{ if}{\kern 2pt}{ \left\{ {{i_1},{i_2},\ldots ,{i_k}} \right\} \in {E(H)}}, \\
 0, {\kern 57pt}\mathrm{ otherwise}. \\
 \end{array} \right.\]
The characteristic polynomial of tensor $\mathcal{A}_H$ is called the characteristic polynomial of the hypergraph $H$ (see \cite{cooper}). For a vector $y=(y_{i_1},y_{i_2},\ldots,y_{i_t})^\mathrm{T} \in \mathbb{C}^t$, let ${y^S} = \prod\nolimits_{i \in S} {{y_i}}$ for $S \subseteq\{i_1,i_2,\ldots, i_t\}$, where $i_1,i_2,\ldots,i_t$ are distinct nonnegative integers. For a vertex $i \in V(H)$,  ${E_i}(H)$ denotes the set of edges incident with vertex $i$. Then
\[{\left( {\mathcal{A}_H{x^{k - 1}}} \right)_i} = \sum\limits_{e\in E_i(H)} {x^{e\setminus \{i\}}} \]
for $x = {\left( {{x_1}, \ldots ,{x_n}} \right)^{\rm{T}}} \in {\mathbb{C}^n}$ and $i \in [n]$.

The characteristic polynomial of graph is an important research topic in spectral graph theory. In 1962, Harary gave the structural parameter representation of the determinant of the adjacency matrix of graphs \cite{ha}. In 1964, Sachs gave the coefficients of characteristic polynomial which is  usually known as Sachs Coefficient Theorem using the result of Harary \cite{Sachs}. In 1971, Harary et al. gave a reduced formula of the characteristic polynomial of graphs with a pendant edge \cite{diedai2}.
\begin{thm}\cite{diedai2}\label{6}
Let $G_v$ denote the graph obtained from $G$ by adding a pendent edge at the vertex $v$. Let $G-v$ denote the graph obtained from $G$ by removing $v$ together with all edges incident to $v$.
Then
\begin{align} \notag
{\phi _{{G_v}}}(\lambda ) = \lambda{\phi _G}(\lambda ) - {\phi _{G - v}}(\lambda ).
\end{align}
\end{thm}
In 1973, Lov\'{a}sz and Pelik\'{a}n gave the characteristic polynomial of paths \cite{Lovasz}.
\begin{thm}\label{L}
\cite{Lovasz} The characteristic polynomial of a path ${P}_m$ of length $m$ is
\[{\phi _{{P_m}}}(\lambda ) = \sum\limits_{q = 0}^{\left\lfloor {\frac{{m + 1}}{2}} \right\rfloor } {{{( - 1)}^q}\mathrm{C}_{m + 1 - q}^q{\lambda ^{m + 1 - 2q}}} ,\]
where $\mathrm{C}_{m + 1 - q}^q$ is a combinatorial number. \\
\end{thm}

In \cite{Cve3}(Page 73 of \cite{Cve3}), the author gave all the distinct eigenvalues of  ${P}_m$.
\begin{thm}\label{Cve3}
\cite{Cve3} The distinct eigenvalues of  ${P}_m$ are  $2\cos \frac{\pi t }{{m + 2}}$,  $t=1,2,\ldots,m+1$.
\end{thm}

In 2012, Cooper and Dutle gave the characteristic polynomial of a $k$-uniform hyperpath with one edge \cite{cooper}. In 2015, Shao et al. gave some properties of the characteristic polynomial of hypergraphs whose spectrum are $k$-symmetric \cite{shao}. In 2015, Cooper and Dutle gave the characteristic polynomial of $3$-uniform hyperstars \cite{cooper xing}. In 2019, Bao et al. gave the characteristic polynomial of $k$-uniform hyperstars and the characteristic polynomial of hypergraphs with a cut vertex under some assumptions \cite{Bao}.

In this paper, we give a reduced formula of the characteristic polynomials of $k$-uniform hypergraphs with pendent edges. And using this formula we give  the explicit characteristic polynomial of hyperpaths.  All distinct eigenvalues of $k$-uniform hyperpath are given.

\section{Preliminary}

In this section, we introduce the Poisson formula and some properties of the resultant which are used in the proof of our main results.

The $k$-uniform hyperpath $P_m^{(k)}$ is the $k$-uniform hypergraph which obtained by adding $k-2$ vertices with degree one to each edge of the path $P_m$.  In 2012, Cooper and Dutle  gave the characteristic polynomial of  $P_1^{(k)}$  \cite{cooper}.

\begin{lem}\cite{cooper}\label{2}
The characteristic polynomial of the $k$-uniform hyperpath $P_1^{(k)}$ is
\[{\phi _{P_1^{(k)}}}(\lambda ) = {\lambda ^{k{{(k - 1)}^{k - 1}} - {k^{k - 1}}}}{\left( {{\lambda ^k} - 1} \right)^{{k^{k - 2}}}}.\]
\end{lem}

In this paper, the \emph{Poisson formula} of resultants is important to compute the characteristic polynomials of hypergraphs.

\begin{lem} (Poisson formula)\label{1}\cite{GtM,I.M.,J.P}
Let $F_0$,$F_1,\ldots,F_n$ be homogeneous polynomials of respective degrees $d_0,\ldots,,d_n$ in $K[x_0,\ldots,x_n]$, where $K$ is an algebraically closed field. For $0\leq i \leq n$, let $\overline {{F_i}}  = {\left. {{F_i}} \right|_{{x_0} = 0}}$ and ${f_i} = {\left. {{F_i}} \right|_{{x_0} = 1}}$.  Let $\mathcal{V}$ be the set of simultaneous zeros of the system of polynomials $f_1,f_2,\ldots,f_n$, that is, $\mathcal{V}$ is the affine variety defined by the polynomials. If ${\rm{Res}}{\left( {\begin{array}{*{20}{c}}
   {\overline {{F_1}} }  \\
    \vdots   \\
   {\overline {{F_n}} }  \\
\end{array}} \right)}\neq 0$, then $\mathcal{V}$ is a zero-dimensional variety (a finite set of points), and
\[{\rm{Res}}\left( {\begin{array}{*{20}{c}}
   {{F_0}}  \\
   {{F_1}}  \\
    \vdots   \\
   {{F_n}}  \\
\end{array}} \right) = {\rm{Res}}{\left( {\begin{array}{*{20}{c}}
   {\overline {{F_1}} }  \\
    \vdots   \\
   {\overline {{F_n}} }  \\
\end{array}} \right)^{{d_0}}}\prod\limits_{p \in V} {{{\left( {{f_0}(p)} \right)}^{m(p)}}} \]

\[{\rm{Res}}\left( \begin{array}{l}
 {F_0} \\
 {F_1} \\
  \vdots  \\
 {F_n} \\
 \end{array} \right) = {\rm{Res}}{\left( {\begin{array}{l}
   {\overline {{F_1}} }  \\
    \vdots   \\
   {\overline {{F_n}} }  \\
\end{array}} \right)^{{d_0}}}{\prod _{p \in {\cal V}}}\left( {{f_0}(p)} \right){^{m(p)}},\]
where $m(p)$ is the multiplicity of a point $p \in \mathcal{V}$.
\end{lem}

\begin{lem}\label{4} (Page 97 and 102 of [1])
Let $F_0$,$F_1,\ldots,F_n$ be homogeneous polynomials of respective degrees $d_0,\ldots,,d_n$ in $K[x_0,\ldots,x_n]$, where $K$ is an algebraically closed field. Then \\
(1) ${\rm{Res}}\left( {\begin{array}{*{20}{c}}
   {{F_0}}  \\
   {{F_1}}  \\
    \vdots   \\
   {\lambda {F_n}}  \\
\end{array}} \right) = {\lambda ^{{d_0}{d_1} \cdots {d_{n - 1}}}}{\rm{Res}}\left( {\begin{array}{*{20}{c}}
   {{F_0}}  \\
   {{F_1}}  \\
    \vdots   \\
   {{F_n}}  \\
\end{array}} \right)$; \\
(2) ${\rm{Res}}\left( {\begin{array}{*{20}{c}}
   {{F_0}}  \\
   {{F_1}}  \\
    \vdots   \\
   {{F_{n - 1}}}  \\
   {x_n^d}  \\
\end{array}} \right) = {\rm{Res}}{\left( {\begin{array}{*{20}{c}}
   {{{\left. {{F_0}} \right|}_{{x_n} = 0}}}  \\
   {{{\left. {{F_1}} \right|}_{{x_n} = 0}}}  \\
    \vdots   \\
   {{{\left. {{F_{n-1}}} \right|}_{{x_n} = 0}}}  \\
\end{array}} \right)^d}{\rm{ }}$.
\end{lem}

For a $k$-uniform hypergraph $H$ with vertices set $V(H)=\{0,1,\ldots,n\}$, let
\begin{align} \notag
F^{H} &= \lambda x^{[k-1]}-\mathcal{A}_Hx^{k-1} =(F^{H}_0,F^{H}_1,\ldots,F^{H}_n)^{\mathrm{T}}, \\ \notag
\overline{F^H}&={{\left. ({F_1^H},{F_2^H},\ldots,{F_n^H})^\mathrm{T} \right|}_{{x_0} = 0}}=(\overline{F_1^H},\overline{F_2^H},\ldots,\overline{F_n^H})^\mathrm{T}, \\  \notag
f_i^H&={{\left. F_i^H \right|}_{{x_0} = 1}}, i=0,1,2,\ldots,n.
\end{align}
where $x=(x_0,x_1,\ldots,x_n)^\mathrm{T} \in \mathbb{C}^{n+1}$. Let $\mathcal{V}^H$ denote the affine variety defined by the polynomials $f_1^H,f_2^H,\ldots,f_n^H$. From Lemma \ref{2},
\begin{align}\label{8261}
\phi_{{H}}(\lambda)= \mathrm{Res}(F^{{H}})=\mathrm{Res}(\overline{F^{{H}}})^{k-1}\prod\limits_{p \in \mathcal{V}^H} {\left( {f_0^{H}(p)} \right)^{m_H(p)}},
\end{align}
where $m_H(p)$ is the multiplicity of a point $p \in \mathcal{V}^H$.

\section{Main results}

In this section, a reduced formula of the characteristic polynomials of $k$-uniform hypergraphs with a pendant edge and the explicit characteristic polynomial of $k$-uniform  hyperpath are given. These results generalize the results given by Harary et al. \cite{diedai2} and Lov\'{a}sz et al. \cite{Lovasz}.

Cooper and Dutle  gave the characteristic polynomial of  $P_1^{(k)}$ via  the trace of tensor \cite{cooper}. Let $V({P_1^{(k)}})=\{0,1,\ldots,k-1\}$. And $\mathcal{V}$ denotes the affine variety defined by the polynomials $f_1^{P_1^{(k)}},f_2^{P_1^{(k)}},\ldots, f_{k-1}^{P_1^{(k)}}$. In order to give the reduced formula of the characteristic polynomials of $k$-uniform hypergraphs with a pendant edge, we first give $m(p)$ for $p \in \mathcal{V}$.

\begin{thm}\label{3}
Let $V({P_1^{(k)}})=\{0,1,\ldots,k-1\}$. And $\mathcal{V}$ denotes the affine variety defined by the polynomials $f_1^{P_1^{(k)}},f_2^{P_1^{(k)}},\ldots, f_{k-1}^{P_1^{(k)}}$. Then $\sum\limits_{0 \ne p \in \mathcal{V} } {m(p)}  = {k^{k - 2}}$ and $m(0)={{{(k - 1)}^{k - 1}} - {k^{k - 2}}}$ for $0 \in \mathcal{V}$.
\end{thm}
\begin{proof}
For the hypergraph $H=P_1^{(k)}$ with one edge $e=\{0,1,\ldots,k-1\}$, $F_i^H=\lambda x_i^{k - 1} - x^{e\setminus \{i\}}$ for $i=0,1,\ldots,k-1$, where $x=(x_0,x_1,\ldots,x_{k-1} )^\mathrm{T}\in \mathbb{C}^k$.  From Eq. (\ref{8261}),
\begin{align}\label{shizi}
{\phi _H}(\lambda )= \mathrm{Res}(F^{H})=\mathrm{Res}(\overline{F^H})^{k-1}\prod\limits_{p \in \mathcal{V}} {\left( {f_0^H(p)} \right)^{m(p)}}.
\end{align}

Since $\overline{F_i^H}={{\left. F_i^H \right|}_{{x_0} = 0}}={{\left. (\lambda x_i^{k - 1} - x^{e\setminus \{i\}} )\right|}_{{x_0} = 0}}=\lambda x_i^{k - 1} $ for $i \in [k-1]$, $\mathrm{Res}( \overline{F^H} )$ is the characteristic polynomial of the $k$-order $(k-1)$-dimension null tensor. From the definition of tensor eigenvalues, we know that the eigenvalues of the null tensor are zero. And from ${\rm{Res}}( \overline{F^H } )$ is monic polynomials of degree ${{{(k - 1)}^{k - 1}}}$, we get
\begin{align} \label{xinde2}
{\rm{Res}}( \overline{F^H } ) = {\lambda ^{{{(k - 1)}^{k - 1}}}}.
\end{align}

For $p = {({{p_1}},{{p_2}}, \ldots ,{{p_{k - 1}}})^\mathrm{T}} \in \mathcal{V}$.  When $p=0$, we have $p^{e\setminus\{0\}}=p_1p_2\cdots p_{k-1}=0$.

When $p\neq0$, we get $f^H_i(p)=\lambda p_i^{k - 1} - p^{e\setminus \{0,i\}}=0$ for all $i\in[k-1]$. Then
\begin{align}\label{lao}
\lambda{p_1^k} = \lambda{p_2^k} =  \cdots  = \lambda{p^k_{k - 1}}= p^{e\setminus\{0\}}
\end{align}
and $\lambda^{k-1}p_1^kp_1^k\cdots p_{k-1}^k=\lambda^{k-1}(p^{e\setminus\{0\}})^k=(p^{e\setminus\{0\}})^{k-1}$. Note that $\lambda$ is an indeterminant of the characteristic polynomials $\phi_H(\lambda)$.  From Eq.(\ref{lao}), we know that  $p^{e\setminus\{0\}}\neq 0$. Then $ p^{e\setminus\{0\}}= \frac{1}{{{\lambda ^{k-1}}}} $. From the above discussion, we obtain
\begin{align}\label{shizi4}
p^{e\setminus\{0\}} = \left\{ \begin{array}{l}
 0,{\kern 50pt} p = 0, \\
 \frac{1}{{{\lambda ^{k-1}}}},{\kern 35pt} p\neq0 .\\
 \end{array} \right.
\end{align}
Hence,
\[{f_0^H}\left( p \right) = \lambda  - p^{e\setminus\{0\}} = \left\{ \begin{array}{l}
 \lambda, {\kern 1pt} {\kern 1pt} {\kern 69pt} p = 0 ,\\
 \lambda  - \frac{1}{{{\lambda ^{k - 1}}}},{\kern 38pt}  p \ne 0 .\\
 \end{array} \right. \]
By Eq. (\ref{shizi}) and Eq. (\ref{xinde2}), we have
\begin{align}\notag \label{91}
{\phi _H}\left( \lambda  \right) &= {\lambda ^{{{(k - 1)}^k}}}\prod\limits_{p \in \mathcal{V}} {\left( {f_0^H(p)} \right)^{m(p)}} \\ \notag
 &= {\lambda ^{{{(k - 1)}^k}}}\prod\limits_{0= p \in \mathcal{V}} {{\lambda ^{m(p)}}} \prod\limits_{0\ne p \in \mathcal{V}} {{{(\lambda  - \frac{1}{{{\lambda ^{k - 1}}}})}^{m(p)}}} \\\notag
&= {\lambda ^{{{(k - 1)}^k}}}{\lambda ^{m(0)}}{(\lambda  - \frac{1}{{{\lambda ^{k - 1}}}})^{\sum\limits_{0\ne p \in \mathcal{V}} {m(p)} }}\\
&={\lambda ^{{{(k - 1)}^k} + m(0) - (k - 1)\sum\limits_{0\ne p \in \mathcal{V}} {m(p)} }}{({\lambda ^k} - 1)^{\sum\limits_{0\ne p \in \mathcal{V}} {m(p)} }}.
\end{align}

Comparing Lemma \ref{2} with Eq.(6), we obtain $\sum\limits_{0\ne p \in \mathcal{V}} {m(p)}  = {k^{k - 2}}$ and $m(0)={{{(k - 1)}^{k - 1}} - {k^{k - 2}}}$.
\end{proof}

Let $H$ be a $k$-uniform hypergraph with vertices set  $V(H)=\{0,1,2,\ldots,n\}$. From Eq.(\ref{8261}), we have
\begin{align}\label{9082}
\phi _{H}(\lambda)={{\rm{Res}}\left( \begin{array}{l}
 \overline{F_{\rm{1}}^H }\\
 \overline{F_2^H} \\
  \vdots  \\
 \overline{F_n^H} \\
 \end{array} \right) ^{k-1}\prod\limits_{p \in \mathcal{V}^H} {\left( {f_0^H(p)} \right)^{m_H(p)}}}.
 \end{align}
For $v \in V(H)$, $H-v$ denotes the $k$-uniform hypergraph obtained from $H$ by removing vertex $v$ and all edges incident to vertex $v$. Without loss of generality, let the vertex $v=0$. Since
\begin{align} \notag
\overline{F^{H}_i}&={{\left.(\lambda x_i^{k - 1} -  \sum\limits_{ e \in E_i(H)} {{x^{e \setminus \{i\} }}}) \right|}_{{x_0} = 0}} \\ \notag
&=\lambda x_i^{k - 1} -  \sum\limits_{\scriptstyle e \in E_i(H) \hfill \atop
  \scriptstyle ~~ 0 \notin e \hfill} {{x^{e \setminus \{i\} }}} \\ \notag
&=\lambda x_i^{k - 1} -  \sum\limits_{ e \in E_i(H-0)} {{x^{e \setminus \{i\} }}}
\end{align}
for $x=(x_0,x_1, \ldots,x_n)^\mathrm{T} \in \mathbb{C}^{n+1}$ and $i \in V(H-0)$, we have ${\rm{Res}}\left( \begin{array}{l}
 \overline{F_{{1}}^H} \\
 \overline{F_2^H }\\
  \vdots  \\
 \overline{F_n^H} \\
 \end{array} \right) = {\phi _{H - 0}}(\lambda )$. Then from Eq.(\ref{9082}), we know
\begin{align} \notag
\phi _{H}(\lambda)&={\phi_{H-0} ^{k-1}\prod\limits_{p \in \mathcal{V}^H} {\left( {f_0^H(p)} \right)^{m_H(p)}}} \\  \notag
&=\phi_{H-0}^{k-1}(\lambda){\prod _{p \in {\cal V}^H}}{\left( \lambda-  \sum\limits_{ e \in E_0(H)} {{p^{e \setminus \{0\} }}} \right)^{m_H(p)}}.
\end{align}

 So for $v \in V(H)$,
\begin{align}\label{908}
\phi _{H}(\lambda) =\phi_{H-v}^{k-1}(\lambda){\prod _{p \in {\cal V}^H}}{\left( \lambda-  \sum\limits_{ e \in E_v(H)} {{p^{e \setminus \{v\} }}} \right)^{m_H(p)}},
\end{align}
where ${\cal V}^H$ is affine variety defined by polynomials ${\left. {{F^H_i}} \right|_{{x_v} = 1}}$ for all $i \in V(H-v)$.

Let
\begin{align}\label{shizi5}
M_H(\lambda,\frac{t}{{{\lambda^{k - 1}}}})= {\prod _{p \in {\cal V}^H}}{\left( \lambda-  \sum\limits_{ e \in E_v(H)} {{p^{e \setminus \{v\} }}} -\frac{t}{{{\lambda^{k - 1}}}} \right)^{m_H(p)}},
\end{align}
where $t$ is a nonnegative integer. And $M_H(\lambda,\frac{t}{{{\lambda^{k-1}}}})$ is called  the ``\emph{$\frac{t}{{{\lambda^{k - 1}}}}$-translational fraction}" of $M_H(\lambda,0)$. By Eq.(\ref{908}) and Eq.(\ref{shizi5}), we know $M_H(\lambda,0)=\frac{{{\phi _H}(\lambda )}}{{\phi _{H - v}^{k - 1}(\lambda )}}$.

We give the reduced formula of the characteristic polynomial of $k$-uniform hypergraphs with pendent edges as follows. When $k=2$, this result is the Theorem \ref{6}.
\begin{thm}\label{5}
Let $H$ be a $k$-uniform hypergraph with $n$ vertices. $H_v$ denotes the $k$-uniform hypergraph obtained from $H$ by adding a pendent edge at the vertex $v$. Let $H-v$ be the $k$-uniform hypergraph obtained from $H$ by removing $v$ and all edges incident to $v$. Then
\begin{align}\notag
{\phi _{{H_v}}}(\lambda ) = {\lambda ^{{{(k - 1)}^{n + k - 1}}}}{\phi _{{H-v}}}{(\lambda)^{{{(k - 1)}^{k}}}}M_H{(\lambda,0)^{{{(k - 1)}^{k - 1}} - {k^{k - 2}}}}M_H(\lambda,\frac{1}{{{\lambda^{k-1}}}}){^{{k^{k - 2}}}},
\end{align}
where $M_H(\lambda ,0) = \frac{{{\phi _H}(\lambda )}}{{\phi _{H - v}(\lambda )^{{{k - 1}}}}}$, $M_H(\lambda,\frac{1}{{{\lambda^{k-1}}}})$ is the $\frac{1}{{{\lambda^{k - 1}}}}$-translational fraction of $M_H(\lambda,0)$.
\end{thm}

\begin{proof}
Without loss of generality, let the vertex $v=0$. Let $V(H)=\{0,1,\ldots,n-1\}$ and $V(H_0)=\{0,1,\ldots,n, n+1,\ldots,n+k-2\}$. Then the pendent edge of $H_0$ is $e_0=\{0,n,n+1,\ldots,n+k-2\}$. Let $\phi_1=\mathrm{Res}\left( {\begin{array}{*{20}{c}}
   {\overline{{F^H}}}  \\
   {\overline{{F^{{e_0}}}}}  \\
\end{array}} \right)$ and \\ $\phi_2={\prod _{p \in {\cal V}^{H_0}}}{(\lambda  - \sum\limits_{ e \in E_0(H)} {{p^{e \setminus \{ 0\} }}}  - {p^{{e_0} \setminus \{ 0\} }})^{m_{H_0}(p)}}$. It follows from Eq.(\ref{8261}) that
\begin{align} \label{xinde}\notag
 {\phi _{{H_{\rm{0}}}}}(\lambda ) &= \mathrm{Res}(\overline{{F^{{H_0}}}})^{k-1}{\prod _{p \in {\cal V}^{H_0}}}{\left( {f_0^{{H_0}}(p)} \right)^{{m_{{H_0}}}(p)}} \\ \notag
 &= \mathrm{Res}(\overline{{F^{{H_0}}}})^{k-1}{\prod _{p \in {\cal V}^{H_0}}}{( \lambda- \sum\limits_{ e \in E_0(H_0)} {{p^{e \setminus \{ 0\} }}} )^{{m_{{H_0}}}(p)}} \\
  &= \mathrm{Res}\left( {\begin{array}{*{20}{c}}
   {\overline{{F^H}}}  \\
   {\overline{{F^{{e_0}}}}}   \notag
\end{array}} \right)^{k-1}{\prod _{p \in {\cal V}^{H_0}}}{(\lambda  - \sum\limits_{ e \in E_0(H)} {{p^{e \setminus \{ 0\} }}}  - {p^{{e_0} \setminus \{ 0\} }})^{m_{H_0}(p)}} \\
&=\phi^{k-1}_1\phi_2 .
 \end{align}

Since $\overline{F_i^{e_0}}={{\left. {F_i^{H_0}} \right|}_{{x_0} = 0}}=\lambda x_i^{k - 1}$ for $ i \in e_0 \setminus \{0\}$, from Lemma \ref{4} (1), we have
\begin{align} \notag
\phi_1=\mathrm{Res}\left( {\begin{array}{*{20}{c}}
   {\overline{{F^H}}}  \\
   {\overline{{F^{{e_0}}}}}  \\
\end{array}} \right)&= {\rm{Res}}\left( {\begin{array}{*{20}{c}}
   {\overline {{F^H}} }  \\
   {\lambda x_n^{k - 1}}  \\
    \vdots   \\
    {\lambda x_{n + k - 3}^{k - 1}}\\
   {\lambda x_{n + k - 2}^{k - 1}}  \\
\end{array}} \right) = {\lambda ^{{{(k - 1)}^{n+k-3}}}}{\rm{Res}}\left( {\begin{array}{*{20}{c}}
   {\overline {{F^H}} }  \\
   {\lambda x_n^{k - 1}}  \\
    \vdots   \\
    {\lambda x_{n + k - 3}^{k - 1}}\\
   {x_{n + k - 2}^{k - 1}}  \\
\end{array}} \right).
\end{align}
Since ${{{\left. {\overline{F^H}} \right|}_{{x_{n+k-2}} = 0}}}=\overline{F^H}$, it follows from that Lemma \ref{4} (2)
that ${\rm{Res}}\left( {\begin{array}{*{20}{c}}
   {\overline {{F^H}} }  \\
   {\lambda x_n^{k - 1}}  \\
    \vdots   \\
    {\lambda x_{n + k - 3}^{k - 1}}\\
   {x_{n + k - 2}^{k - 1}}  \\
\end{array}} \right)= {\rm{Res}}\left( {\begin{array}{*{20}{c}}
   {\overline {{F^H}} }  \\
   {\lambda x_n^{k - 1}}  \\
    \vdots   \\
    {\lambda x_{n + k - 3}^{k - 1}}\\
\end{array}} \right)^{k-1} $.
Then $\phi_1={\lambda ^{{{(k - 1)}^{n+k-3}}}}{\rm{Res}}\left( {\begin{array}{*{20}{c}}
   {\overline {{F^H}} }  \\
   {\lambda x_n^{k - 1}}  \\
    \vdots   \\
    {\lambda x_{n + k - 3}^{k - 1}}\\
\end{array}} \right)^{k-1} .$
By repeating the above process, we obtain $\phi_1={\lambda ^{{{(k - 1)}^{n+k-2}}}}{\rm{Res}}({\overline {{F^H}} })^{(k-1)^{k-1}} $.

From $\overline{F^H_i}={{\left. F_i^{H_0} \right|}_{{x_0} = 0}}=\lambda x_i^{k - 1} -  \sum\limits_{\scriptstyle  e \in E_i(H) \hfill \atop
  \scriptstyle ~~ 0 \notin e \hfill} {{x^{e \setminus \{i\} }}}=\lambda x_i^{k - 1} -  \sum\limits_{ e \in E_i(H-0)} {{x^{e \setminus \{i\} }}}$ for $i \in V(H-0)$,  it yields that $\mathrm{Res}(\overline{F^H})=\phi_{H-0}(\lambda)$. We obtain
\begin{align}
\phi_1={\rm{Res}}\left( {\begin{array}{*{20}{c}}
   {\overline {{F^H}} }  \\
   {\overline {{F^{{e_0}}}} }  \\
\end{array}} \right) = {\lambda ^{{{(k - 1)}^{n + k - 2}}}}{\phi _{H - 0}}{(\lambda )^{{{(k - 1)}^{k - 1}}}}. \notag
\end{align}

Note that ${\cal V}^{H_0}={\cal V}^{H}\times {\cal V}^{e_0}$.  For $p \in {\cal V}^{H_0}$, we have vector $p = \left( {\begin{array}{*{20}{c}}
   q  \\
   r  \\
\end{array}} \right)
$, where $q \in {\cal V}^{H}$, $r \in {\cal V}^{e_0} $. Let $r =(r_n,r_{n+1},\ldots,r_{n+k-2})^\mathrm{T}\in{\cal V}^{e_0}$. From Eq. (\ref{shizi4}), we know that
\begin{align}\notag
r^{e_0\setminus\{0\}}={r_n} \cdots {r_{n + k - 2}} = \left\{ \begin{array}{l}
 0,{\kern 50pt}r = 0, \\
 \frac{1}{{{\lambda ^{k-1}}}},{\kern 35pt} r \neq0 .\\
 \end{array} \right.
\end{align}
By Lemma \ref{3}, we have $m_{e_0}(0)={{{(k - 1)}^{k - 1}} - {k^{k - 2}}}$ for $0 \in \mathcal{V}^{e_0}$ and $\sum\limits_{0\ne r \in \mathcal{V}^{e_0} } {m_{e_0}(r)}  = {k^{k - 2}}$. Hence
\begin{align} \notag
 &\phi_2= \prod\limits_{p \in \mathcal{V}^{H_0}} {{{(\lambda  - \sum\limits_{ e \in E_0(H)} {{p^{e\backslash \{ 0\} }}}  - {p^{e_0\backslash \{ 0\} }})}^{m_{H_0}(p)}}}
 = \prod\limits_{\scriptstyle q \in {\cal V}^{H} \hfill \atop
  \scriptstyle r \in {\cal V}^{e_0} \hfill} {{{(\lambda  - \sum\limits_{ e\in E_0(H)} {{q^{e\backslash \{ 0\} }}}  - {r^{e_0\backslash \{ 0\} }})}^{{m_{H}}(q){m_{e_0}}(r)}}}  \\ \notag
  &=\prod\limits_{\scriptstyle q \in {\cal V}^{H} \hfill \atop
  \scriptstyle  0= r \in \mathcal{V}^{e_0} \hfill} {{{(\lambda \!-\!\sum\limits_{e \in E_0(H)} \! {{q^{e\backslash \{ 0\} }}}   \!- \! {r^{e_0\backslash \{ 0\} }})}^{{m_{H}}(q){m_{e_0}}(r)}}} \prod\limits_{\scriptstyle q \in {\cal V}^{H} \hfill \atop
  \scriptstyle 0 \neq r \in \mathcal{V}^{e_0} \hfill} {{{(\lambda  \! - \! \sum\limits_{e \in E_0(H)} {{q^{e\backslash \{ 0\} }}}  \! - \! {r^{e_0\backslash \{ 0\} }})}^{{m_{H}}(q){m_{e_0}}(r)}}} \\ \notag
 & = \prod\limits_{ q \in {\cal V}^{H} } {{{(\lambda  - \sum\limits_{e  \!\in \! E_0(H)} {{q^{e\backslash \{ 0\} }}} )}^{{m_{H}}(q)({{(k - 1)}^{k - 1}} - {k^{k - 2}})}}}  \prod\limits_{ q \in {\cal V}^{H}} {{{(\lambda  - \sum\limits_{e \in E_0(H)} {{q^{e\backslash \{ 0\} }}}  - \frac{1}{{{\lambda ^{k - 1}}}})}^{{m_{H}}(q) {k^{k - 2}}}}} .\notag
\end{align}
By Eq. (\ref{shizi5}), we have
\[M_H(\lambda,0)= \frac{{{\phi _H}(\lambda )}}{{{\phi _{H - 0}}{{(\lambda )}^{k - 1}}}}=\prod\limits_{ q \in {\cal V}^{H} } {{{(\lambda  - \sum\limits_{ e \in E_0(H)} {{q^{e\backslash \{ 0\} }}} )}^{{m_{H}}(q)}}}\]
 and \[ M_H(\lambda,\frac{1}{{{\lambda^{k-1}}}})=\prod\limits_{q \in {\cal V}^{H}}{{{(\lambda  - \sum\limits_{ e \in E_0(H) } {{q^{e\backslash \{ 0\} }}}  - \frac{1}{{{\lambda ^{k - 1}}}})}^{{m_{H}}(q)}}} .\]
Then
\[\phi_2=\prod\limits_{p \in \mathcal{V}^{H_0}} {{{(\lambda  - \sum\limits_{e \in E_0(H)} {{p^{e\backslash \{ 0\} }}}  - {p^{e_0\backslash \{ 0\} }})}^{m_{H_0}(p)}}}=M_H{(\lambda,0)^{{{(k - 1)}^{k - 1}} - {k^{k - 2}}}}{M_H(\lambda,\frac{1}{{{\lambda^{k-1}}}})^{{k^{k - 2}}}}.\]
Substituting $\phi_1$ and $\phi_2$ into Eq.(\ref{xinde}), the proof is completed.
\end{proof}

In Theorem \ref{5}, when $k=2$, $H$ is a graph with $n$ vertices. ${\phi _H}(\lambda )$ and ${{\phi _{H - v}}(\lambda )}$ are polynomials of degree $n$ and $n-1$, respectively. It follows from Eq. (\ref{shizi5}) that
\[M_H(\lambda,0) = \frac{{{\phi _H}(\lambda )}}{{{\phi _{H - v}}(\lambda )}} = \lambda  + \frac{{{\phi _H}(\lambda ) - \lambda {\phi _{H - v}}(\lambda )}}{{{\phi _{H - v}}(\lambda )}},\]
\begin{align} \notag
 M_H(\lambda,\frac{1}{{{\lambda}}}) &= \lambda  + \frac{{{\phi _H}(\lambda ) - \lambda {\phi _{H - v}}(\lambda )}}{{{\phi _{H - v}}(\lambda )}} - \frac{1}{\lambda } = \frac{{{\phi _H}(\lambda )}}{{{\phi _{H - v}}(\lambda )}} - \frac{1}{\lambda } .\notag
\end{align}
Then \begin{align} \notag
{\phi _{{H_v}}}(\lambda ) &= \lambda {\phi _{H - v}}(\lambda )\left( {\frac{{{\phi _H}(\lambda )}}{{{\phi _{H - v}}(\lambda )}} - \frac{1}{\lambda }} \right)\\ \notag
&= \lambda {\phi _H}(\lambda ) - {\phi _{H - v}}(\lambda ).
\end{align}Hence, the Theorem \ref{5} is the Theorem \ref{6} when $k=2$.

Let $H^{s}_v$ denote the $k$-uniform hypergraph obtained from  hypergraph $H$ by adding $s$ pendent edges at the vertex $v$. By Theorem \ref{5}, we  give a reduced formula of the characteristic polynomial of $H^{s}_v$.

\begin{thm}
Let $H$ be a $k$-uniform hypergraph with $n$ vertices. $H^s_v$ denotes the $k$-uniform hypergraph obtained from $H$ by adding $s$ pendent edges at the vertex $v$. Let $H-v$ be the $k$-uniform hypergraph obtained from $H$ by removing $v$ together with all edges incident to $v$. Then
\begin{align}\notag
{\phi _{H_v^s}}(\lambda ) = {\lambda ^{s{{(k - 1)}^{n + s(k - 1)}}}}{\phi _{H - v}}{(\lambda )^{{{(k - 1)}^{s(k - 1) + 1}}}}\prod\limits_{t = 0}^s {{{\left( {M_H(\lambda ,\frac{t}{{{\lambda ^{k - 1}}}})} \right)}^{\mathrm{C}_s^t{{K_1 }^{s - t}}{{K_2}^t}}}} ,
\end{align}
where $M_H(\lambda,\frac{t}{{{\lambda^{k-1}}}})$ is $\frac{t}{{{\lambda^{k - 1}}}}$-translational fraction of $M_H(\lambda,0) = \frac{{{\phi _H}(\lambda )}}{{\phi _{H - v}(\lambda )^{{{k - 1}}}}}$, $K_1 = {(k - 1)^{k - 1}} - {k^{k - 2}}$, $K_2 = {k^{k - 2}}$ and $\mathrm{C}_s^t$ is a combinatorial number.
\end{thm}

\begin{proof}
From Theorem \ref{5}, we have
\begin{align}\notag
{\phi _{{H_v}}}(\lambda ) &= {\lambda ^{{{(k - 1)}^{n + k - 1}}}}{\phi _{{H-v}}}{(\lambda)^{{{(k - 1)}^{k}}}}M_H{(\lambda,0)^{{{(k - 1)}^{k - 1}} - {k^{k - 2}}}}M_H(\lambda,\frac{1}{{{\lambda^{k-1}}}}){^{{k^{k - 2}}}} \\ \notag
& = {\left( {{\phi _{{H_v} - v}}(\lambda )} \right)^{k - 1}}{M_H}{(\lambda,0 )^{K_1}}{M_H}{(\lambda ,\frac{1}{{{\lambda ^{k - 1}}}})^{K_2}}.
\end{align}
Then
\begin{align} \label{911}
{M_{{H_v}}}(\lambda,0 ) = \frac{{{\phi _{{H_v}}}(\lambda )}}{{{{\left( {{\phi _{{H_v} - v}}(\lambda )} \right)}^{k - 1}}}} = {M_H}{(\lambda,0 )^{K_1}}{M_H}{(\lambda ,\frac{1}{{{\lambda ^{k - 1}}}})^{K_2}},
\end{align}
and
${M_{{H_v}}}(\lambda ,\frac{1}{{{\lambda ^{k - 1}}}}) = {M_H}{(\lambda ,\frac{1}{{{\lambda ^{k - 1}}}})^{{K_1}}}{M_H}{(\lambda ,\frac{2}{{{\lambda ^{k - 1}}}})^{{K_2}}}$. From Eq.(\ref{911}), we get
\begin{align} \notag
{M_{H_v^2}}(\lambda,0 ) &= {M_{{H_v}}}{(\lambda ,0)^{{K_1}}}{M_{{H_v}}}{(\lambda ,\frac{1}{{{\lambda ^{k - 1}}}})^{{K_2}}} \\ \notag
 &= \prod\limits_{t = 0}^2 {{M_H}{{(\lambda ,\frac{t}{{{\lambda ^{k - 1}}}})}^{\mathrm{C}^t_2 K_{_1}^{2 - t}K_{_2}^t}}}.
\end{align}
By induction, we obtain
\[{M_{H_v^s}}(\lambda )=\frac{{{\phi _{H_v^s}}(\lambda )}}{{{\phi _{H_v^s - v}}{{(\lambda )}^{k - 1}}}} = \prod\limits_{t = 0}^s {{M_H}{{(\lambda ,\frac{t}{{{\lambda ^{k - 1}}}})}^{C^t_s K_{_1}^{s - t}K_{_2}^t}}}\] for $s\geq 1$.
Then
\begin{align} \notag
{\phi _{H_v^s}}(\lambda ) &= {\phi _{H_v^s - v}}{(\lambda )^{k - 1}}\prod\limits_{t = 0}^s {{M_H}{{(\lambda ,\frac{t}{{{\lambda ^{k - 1}}}})}^{\mathrm{C}^t_s K_{_1}^{s - t}K_{_2}^t}}} \\ \notag
 & = {\lambda ^{s{{(k - 1)}^{n + s(k - 1)}}}}{\phi _{H - v}}{(\lambda )^{{{(k - 1)}^{s(k - 1) + 1}}}}\prod\limits_{t = 0}^s {{M_H}{{(\lambda ,\frac{t}{{{\lambda ^{k - 1}}}})}^{\mathrm{C}^t_s K_{_1}^{s - t}K_{_2}^t}}} .
\end{align}

\end{proof}

We use Theorem \ref{5} to get the characteristic polynomial and all distinct eigenvalues of  $P_m^{(k)}$.  And we express the characteristic polynomial of  $P_m^{(k)}$  by the characteristic polynomial of path. For convenience, we prove it by induction.

\begin{thm}\label{lu}
The characteristic polynomial of the $k$-uniform hyperpath $P_m^{(k)}$ of length $m$ is
\[{\phi _{{\cal P}_m^{(k)}}}(\lambda ) = \prod\limits_{j = 0}^{m } {{\phi _{{P_j}}}{{({\lambda ^{\frac{k}{2}}})}^{^{a(j,m)}}}} ,\]
where ${\phi _{{P_j}}}(\lambda ) = \sum\limits_{t = 0}^{\left\lfloor {\frac{{j + 1}}{2}} \right\rfloor } {{{( - 1)}^t}\mathrm{C}_{j + 1 - t}^t{\lambda ^{j + 1 - 2t}}} $ is the characteristic polynomial of the path $P_j$,
\[a(j,m) = \left\{ \begin{array}{l}
 {K_2^m}{\kern 1pt} {\kern 1pt} {\kern 1pt} ,~~j = m, \\
 \left( {(m - j + 1)K_1 + 2K_2} \right)K_1{K_2^j}{(k-1)^{(m - j - 2)(k-1)}},~~1 \le j \le m - 1, \\
 \frac{2}{k}\left[ {m(k - 1) + 1} \right]{(k-1)^{m(k-1)}}  - \sum\limits_{r = 1}^{m } {(r + 1)a(r,m),~j = 0,}  \\
 \end{array} \right.\]
$K_1 = {(k - 1)^{k - 1}} - {k^{k - 2}}$ and $K_2 = {k^{k - 2}}$.
\end{thm}

\begin{proof}
By the reduced formula in Theorem \ref{5} and induction, we give this proof. When $m=1$, it follows from Lemma \ref{2} that
\begin{align} \notag
{\phi _{P_1^{(k)}}}(\lambda ) &= {\lambda ^{k{{(k - 1)}^{k - 1}} - {k^{k - 1}}}}{\left( {{\lambda ^k} - 1} \right)^{{k^{k - 2}}}} \\ \notag
&= (\lambda^k)^{K_1}   {\left( {{\lambda ^k} - 1} \right)^{K_2}} \\   \notag
&= \prod\limits_{j = 0}^{1 } {{\phi _{{P_j}}}{{({\lambda ^{\frac{k}{2}}})}^{^{a(j,1)}}}}.
\end{align}
Let $V({P_1^{(k)}})=\{0,1,\ldots,k-1\}$. Then
\[{\phi _{{\cal P}_1^{(k)}{\rm{ - }}{0}}}(\lambda ) = {\lambda ^{{{(k - 1)}^{k - 1}}}} = {\lambda ^{K_1+K_2}}.\]
So
\begin{align} \notag
&M_{{{\cal P}_1^{(k)}}}(\lambda,0 )= \frac{{{\phi _{{\cal P}_1^{(k)}}}(\lambda )}}{{{\phi _{{\cal P}_1^{(k)}{\rm{ - }}{0}}}{{(\lambda )}^{k - 1}}}} = {\lambda ^{K_1}}{\left( {\lambda  - \frac{1}{{{\lambda ^{k - 1}}}}} \right)^{K_2}}  \notag
\end{align}
and \begin{align} \notag
 &M_{{{\cal P}_1^{(k)}}}(\lambda,\frac{1}{{{\lambda ^{k - 1}}}}) = {\left( {\lambda  - \frac{1}{{{\lambda ^{k - 1}}}}} \right)^{K_1}}{\left( {\lambda  - \frac{2}{{{\lambda ^{k - 1}}}}} \right)^{K_2}}.  \notag
\end{align}
Then from the reduced formula in Theorem \ref{5}, we get
\begin{align}\notag
{\phi _{{\cal P}_2^{(k)}}}(\lambda ) &= {\lambda ^{(2(k - 1) + 1){{(K_1 + K_2)}^2} - 2K_1K_2k -K_2^2k}}{\left( {{\lambda ^k} - 1} \right)^{2K_1K_2}}{\left( {{\lambda ^k} - 2} \right)^{{K_2^2}}} \\ \notag
&= \prod\limits_{j = 0}^{2 } {{\phi _{{P_j}}}{{({\lambda ^{\frac{k}{2}}})}^{^{a(j,2)}}}} .\notag
\end{align}
Assume that for $m_0 \ge 2$, $${\phi _{{\cal P}_{m_0}^{(k)}}}(\lambda ) = \prod\limits_{j = 0}^{m_0 } {{\phi _{{P_j}}}{{({\lambda ^{\frac{k}{2}}})}^{^{a(j,m_0)}}}}$$
and
$${\phi _{{\cal P}_{m_0-1}^{(k)}}}(\lambda ) = \prod\limits_{j = 0}^{m_0-1 } {{\phi _{{P_j}}}{{({\lambda ^{\frac{k}{2}}})}^{^{a(j,m_0-1)}}}}.$$
Let  $v$ be the pendent vertex  in ${P_{m_0}^{(k)}}$. Then
\begin{align}\notag
{\phi _{P_{{m_0}}^{(k)} - v}}(\lambda )&= {\lambda ^{\left( {k - 2} \right){{\left( {k - 1} \right)}^{{m_0}(k - 1) - 1}}}}{\left( {{\phi _{P_{{m_0} - 1}^{(k)}}}(\lambda )} \right)^{{{\left( {k - 1} \right)}^{k - 2}}}}\\ \notag
&={\lambda ^{\left( {k - 2} \right){{\left( {k - 1} \right)}^{{m_0}(k - 1) - 1}}}}{\left(  \prod\limits_{j = 0}^{m_0-1 } {{\phi _{{P_j}}}{{({\lambda ^{\frac{k}{2}}})}^{^{a(j,m_0-1)}}}}  \right)^{{{\left( {k - 1} \right)}^{k - 2}}}}.
\end{align}
So
\begin{align}\notag \label{11}
&M_{P_{{m_0}}^{(k)}}(\lambda,0 ) = \frac{{{\phi _{P_{{m_0}}^{(k)}}}(\lambda )}}{{{\phi _{P_{{m_0}}^{(k)} - v}}{{(\lambda )}^{k - 1}}}},\\
&= {\lambda ^{K_1{{\left( {K_1 + K_2} \right)}^{{m_0} - 1}}}}{\prod\limits_{j = 1}^{{m_0} - 1} {\left( {{\lambda ^{\frac{{2 - k}}{2}}}\frac{{{\phi _{{P_j}}}({\lambda ^{\frac{k}{2}}})}}{{{\phi _{{P_{j - 1}}}}({\lambda ^{\frac{k}{2}}})}}} \right)} ^{{{(K_1 + K_2)}^{{m_0} - 1 - j
}}{K_2^j}K_1}}{\left( {{\lambda ^{\frac{{2 - k}}{2}}}\frac{{{\phi _{{P_{{m_0}}}}}({\lambda ^{\frac{k}{2}}})}}{{{\phi _{{P_{{m_0} - 1}}}}({\lambda ^{\frac{k}{2}}})}}} \right)^{{K_2^{{m_0}}}}}.
\end{align}

By  Theorem \ref{L}, we have ${\phi _{P_{{0}}}}(\lambda )=\lambda$, ${\phi _{P_{{1}}}}(\lambda )=\lambda^2-1$. By  Theorem \ref{6}, we have
\begin{align}\label{10}
{\phi _{{P_{j}}}}(\lambda ) = \lambda {\phi _{{P_{j-1}}}}(\lambda ) - {\phi _{{P_{j - 2}}}}(\lambda ) .
\end{align}
It follows from Eq.(\ref{10}) that ${\phi _{P_{{-1}}}}(\lambda )=1$ when $j=1$.
Replace $\lambda$ with ${\lambda ^{\frac{k}{2}}}$ in Eq. (\ref{10}), we get ${\phi _{{P_{j}}}}({\lambda ^{\frac{k}{2}}} ) = {\lambda ^{\frac{k}{2}}}{\phi _{{P_{j-1}}}}({\lambda ^{\frac{k}{2}}} ) - {\phi _{{P_{j - 2}}}}({\lambda ^{\frac{k}{2}}} )$. Then
\[\frac{{{\phi _{{P_j
}}}( {\lambda ^{\frac{k}{2}}})}}{{{\phi _{{P_{j - 1}}}}( {\lambda ^{\frac{k}{2}}} )}} = {\lambda ^{\frac{k}{2}}}  - \frac{{{\phi _{{P_{j - 2}}}}( {\lambda ^{\frac{k}{2}}})}}{{{\phi _{{P_{j - 1}}}}( {\lambda ^{\frac{k}{2}}}  )}}.\]
Therefore,
\begin{align}\label{906}
 {\lambda ^{\frac{{2 - k}}{2}}}\frac{{{\phi _{{P_{j}}}}({\lambda ^{\frac{k}{2}}})}}{{{\phi _{{{\mathop{\rm P}\limits} _{j-1}}}}({\lambda ^{\frac{k}{2}}})}} = \lambda  - \frac{{{\phi _{{P_{j - 2}}}}({\lambda ^{\frac{k}{2}}})}}{{{\lambda ^{\frac{{k - 2}}{2}}}{\phi _{{P_{j-1}}}}({\lambda ^{\frac{k}{2}}})}}.
\end{align}
From Eq. (\ref{11}) and Eq. (\ref{906}), it yields that
\begin{align}\notag
&M_{P_{{m_0}}^{(k)}}(\lambda,0 )= \\ \notag
&{\lambda ^{K_1{{\left( {K_1 + K_2} \right)}^{{m_0} - 1}}}}{\prod\limits_{j = 1}^{{m_0} - 1} {\left( \lambda  - \frac{{{\phi _{{P_{j - 2}}}}({\lambda ^{\frac{k}{2}}})}}{{{\lambda ^{\frac{{k - 2}}{2}}}{\phi _{{P_{j-1}}}}({\lambda ^{\frac{k}{2}}})}} \right)} ^{{{(K_1 + K_2)}^{{m_0} - 1 - j}}K_1{K_2^j}}}{\left( \lambda  - \frac{{{\phi _{{P_{m_0 - 2}}}}({\lambda ^{\frac{k}{2}}})}}{{{\lambda ^{\frac{{k - 2}}{2}}}{\phi _{{P_{m_0-1}}}}({\lambda ^{\frac{k}{2}}})}} \right)^{{K_2^{{m_0}}}}}.
\end{align}
Next, we give the representation of $M_{P_{{m_0}}^{(k)}} \left( {\lambda,\frac{1}{{{\lambda ^{k - 1}}}}} \right)$. Since
\begin{align}\notag
 \lambda  - \frac{{{\phi _{{P_{j - 2}}}}({\lambda ^{\frac{k}{2}}})}}{{{\lambda ^{\frac{{k - 2}}{2}}}{\phi _{{P_{j-1}}}}({\lambda ^{\frac{k}{2}}})}} - \frac{1}{{{\lambda ^{k - 1}}}}& = {\lambda ^{\frac{{2 - k}}{2}}}\frac{{{\phi _{{P_{j}}}}({\lambda ^{\frac{k}{2}}})}}{{{\phi _{{{\mathop{ P}\limits} _{j-1}}}}({\lambda ^{\frac{k}{2}}})}} - \frac{1}{{{\lambda ^{k - 1}}}} \\ \notag
& =\frac{{{\lambda ^{\frac{k}{2}}}{\phi _{{P_j}}}({\lambda ^{\frac{k}{2}}}) - {\phi _{{P_{j- 1}}}}({\lambda ^{\frac{k}{2}}})}}{{{\lambda ^{k-1}}{\phi _{{P_{j - 1}}}}({\lambda ^{\frac{k}{2}}})}} \\ \notag
&=\frac{{{{\phi _{{P_{j+ 1}}}}({\lambda ^{\frac{k}{2}}})}}}{{{\lambda ^{k-1}}{\phi _{{P_{j - 1}}}}({\lambda ^{\frac{k}{2}}})}},
\end{align}
we have
\begin{align}\notag
&M_{P_{{m_0}}^{(k)}} \left( {\lambda,\frac{1}{{{\lambda ^{k - 1}}}}} \right) = \\ \notag
&{(\frac{{{\lambda ^k} - 1}}{{{\lambda ^{k - 1}}}})^{K_1{{\left( {K_1 + K_2} \right)}^{{m_0} - 1}}}}{\prod\limits_{j = 1}^{{m_0} - 1} {\left( {\frac{{{\phi _{{P_{j + 1}}}}({\lambda ^{\frac{k}{2}}})}}{{{\lambda ^{k - 1}}{\phi _{{P_{j - 1}}}}({\lambda ^{\frac{k}{2}}})}}} \right)} ^{{{(K_1 + K_2)}^{{m_0} - 1 - j}}K_1{K_2^j}}}{\left( {\frac{{{\phi _{{P_{{m_{0 + 1}}}}}}({\lambda ^{\frac{k}{2}}})}}{{{\lambda ^{k - 1}}{\phi _{{P_{{m_0} - 1}}}}({\lambda ^{\frac{k}{2}}})}}} \right)^{{K_2^{{m_0}}}}}.
\end{align}
Then from Theorem \ref{5}, we  obtain
\begin{align} \notag
{\phi _{{\cal P}_{m_0+1}^{(k)}}}(\lambda )&=\lambda^{(k-1)^{(m_0+1)(k-1)+1}}\left(\phi_{P_{m_0}^{(k)}-v}(\lambda)\right)^{(k-1)^k}M_{P_{{m_0}}^{(k)}} \left( {\lambda} \right)^{K_1}M_{P_{{m_0}}^{(k)}} \left( {\lambda,\frac{1}{{{\lambda ^{k - 1}}}}} \right)^{K_2}\\ \notag
&= \prod\limits_{j = 0}^{m_0+1 } {{\phi _{{P_j}}}{{({\lambda ^{\frac{k}{2}}})}^{^{a(j,m_0+1)}}}} .
\end{align}
By  induction, we get
\[{\phi _{{\cal P}_m^{(k)}}}(\lambda ) = \prod\limits_{j = 0}^{m } {{\phi _{{P_j}}}{{({\lambda ^{\frac{k}{2}}})}^{^{a(j,m)}}}} .\]

\end{proof}

 From Theorem \ref{Cve3}, we know that all the distinct eigenvalues of a path ${P}_m$ are  $2\cos \frac{\pi t }{{m + 2}}$,  $t=1,2,\ldots,m+1$ i.e. ${\phi _{{P_m}}}(\lambda ) = \prod\limits_{t=1}^{m+1} {\left( {\lambda  - 2\cos \frac{\pi }{{m + 2}}t} \right)}$. Then ${\phi _{{P_m}}}({\lambda ^{\frac{k}{2}}}) = \prod\limits_{t = 1}^{m + 1} {\left( {{\lambda ^{\frac{k}{2}}} - 2\cos \frac{\pi }{{m + 2}}t} \right)}$. From Theorem \ref{lu}, we directly get the following result.

\begin{thm}\label{8}
 The distinct eigenvalues of the $k$-uniform hyperpath $P_m^{(k)}$ are  the different numbers of ${\left( {2\cos \frac{\pi }{{j+2 }}t} \right)^{\frac{2}{k}}}{e^{\mathbf{i}\frac{{2\pi}}{k}\theta }}$  for all $j \in [m]$, $t \in [j+1]$ and $\theta \in [k]$, where $\mathbf{i}^2=-1$.
\end{thm}

Let $\rho(P_m^{(k)})$ be the the spectral radius of $P_m^{(k)}$. In 2016, Lu and Man proved that $\mathop {\lim }\limits_{m \to \infty } \rho(P_m^{(k)}) = \sqrt[k]{4}$ (see \cite{Lu}). From Theorem \ref{8}, we know that $\rho(P_m^{(k)} )={\left( {2\cos \frac{\pi }{{m + 2}}} \right)^{\frac{2}{k}}} $.

\section*{Acknowledgement}
 Supported  by the National Natural Science Foundation of China (No. 11801115 and No. 11601102), the Natural Science Foundation of the Heilongjiang Province (No. QC2018002) and the Fundamental Research Funds for the Central Universities.

\end{document}